\documentclass[11pt]{article}
\usepackage[margin=2cm]{geometry}

\usepackage{xcolor,soul,framed} %,caption
\usepackage[ruled,vlined]{algorithm2e}
\colorlet{shadecolor}{yellow}
\usepackage{graphicx}
\DeclareGraphicsExtensions{.pdf,.jpeg,.png}

\usepackage{tikz}
\usetikzlibrary{shapes.geometric, arrows.meta, positioning}

\usepackage[cmex10]{amsmath}
\usepackage{array}
\usepackage{mdwmath}
\usepackage{mdwtab}
\usepackage{eqparbox}
\usepackage{url}
\usepackage[numbers]{natbib}
\usepackage{booktabs}
\usepackage{enumitem}
\usepackage{amsmath}
\usepackage{amssymb}
\usepackage{amsthm}

\usepackage{multirow}
\usepackage{makecell} % if you're still breaking lines

\usepackage{mathtools, cuted}
\usepackage{lipsum, color}

\usepackage{pifont}

\graphicspath{{figures/}}

\newtheorem{proposition}{Proposition}

\begin{document}

\title{Best-Worst Disaggregation: An approach to the preference
disaggregation problem}

\author{Matteo Brunelli$^{1}$, Fuqi Liang$^2$, and Jafar Rezaei$^3$ \\[0.75em]
\parbox{0.85\textwidth}{\footnotesize
$^{1}\;$Department of Industrial Engineering, University of Trento, Italy\\
$^{2}\;$School of Management, Zhejiang University, and Hangzhou Global Scientific and Technological Innovation Center, China\\
$^{3}\;$Department of Engineering Systems and Services, Delft University of Technology, The Netherlands}
}
\date{}

\maketitle

% === ABSTRACT ====================================================================
% =================================================================================
\begin{abstract}
%\boldmath
%Preference disaggregation methods in Multi-Criteria Decision Analysis (MCDA) often encounter challenges related to inconsistency and cognitive biases when deriving a value function from experts’ holistic preferences.
This paper introduces the Best-Worst Disaggregation (BWD) method, a novel approach that integrates the principles of the Best-Worst Method (BWM) into the disaggregation framework to enhance the consistency and reliability of derived preference models.
%BWD employs the “consider-the-opposite” strategy from BWM, allowing experts to provide two opposite pairwise comparison vectors of alternatives. This approach reduces cognitive load and mitigates anchoring bias, possibly leading to more reliable criteria weights and attribute value functions.
An optimization model is formulated to determine the most suitable additive value function that fits the preferences expressed by a decision-maker or expert. The method incorporates a consistency analysis to quantify and improve the reliability of the judgments and is extended to handle interval-valued preferences. We also developed an approach to identify a reference set, which is used for pairwise comparisons to elicit the value functions and weights. A case study in logistics performance evaluation demonstrates the practicality and effectiveness of BWD. By coupling a structured human-elicitation protocol with an inference engine and a consistency-feedback loop, BWD provides a human-centered building block for evidence-sensitive decision-support systems.
\end{abstract}

% === KEYWORDS ====================================================================
% =================================================================================
\noindent\textbf{Keywords:} Multi-criteria decision analysis; Best-Worst method; Preference disaggregation; Consistency analysis; Imprecise preference information.\par
\medskip
\begin{center}
\setlength{\fboxsep}{8pt}
\fbox{\begin{minipage}{0.9\textwidth}
\textbf{Publication note:} The final editorial version of this paper is published as: Brunelli, M., Liang, F., and Rezaei, J. (2026), ``Best-Worst Disaggregation: an approach to the preference disaggregation problem,'' \emph{IEEE Transactions on Systems, Man and Cybernetics: Systems}, 56(11), 7308--7321.
\end{minipage}}
\end{center}
\medskip

% ====================================================================
% ====================================================================
% ====================================================================

% === I. INTRODUCTION =============================================================
% =================================================================================
\section{Introduction}

The field of Multi-Criteria Decision Analysis (MCDA) encompasses a wide range of methods designed to aid decision-makers and experts \footnote{For brevity, we use the term ``expert'' to refer to both experts and decision-makers throughout this study.} in evaluating a set of alternatives characterized by a set of attributes/criteria \citep{%Haimes1985,
GrecoEtAl2016, GrecoSlowinskiWallenius2025}. Within MCDA, value-based methods, such as Multi-Attribute Value Theory (MAVT) \citep{KeeneyRaiffa1993}, Analytic Hierarchy Process (AHP) \citep{Saaty1977}, and Best-Worst Method (BWM) \citep{Rezaei2015}, aim to assign a numerical value to each alternative. In these methods, experts assess the relative importance of criteria and the performance of alternatives on these criteria, often using pairwise comparisons to derive weights and scores. The scores are then aggregated to produce an overall ranking, emphasizing the total utility or value derived from each alternative.

MCDA methods are commonly distinguished by their preference elicitation strategy, falling into aggregation and disaggregation paradigms \citep{Jacquet2001}: in the aggregation paradigm, the analyst first asks the expert to specify the parameters of a preference model (e.g., criteria weights and marginal value functions) and then aggregates them to rank the alternatives, whereas in the disaggregation paradigm, the direction is reversed: the expert provides overall judgments on a few alternatives, and the parameters of the preference model are inferred from these judgments. Aggregation methods, such as MAVT, elicit \emph{value functions} 
(mathematical mappings that translate the performance of an alternative on each criterion into a comparable numerical score)
and trade-offs among the objectives to evaluate and rank the alternatives. Disaggregation methods, including the UTA (UTilités Additives) family \citep{SiskosEtAl2016, DuarteEtAl2025}, on the other hand, are appropriate when an expert can provide \emph{holistic} preferential information---that is, overall comparative judgments on entire alternatives rather than on individual criteria---on a \emph{reference set} of alternatives, i.e., a small subset of well-known alternatives on which the expert feels confident enough to express overall judgments. These methods will then help infer the underlying value functions and trade-offs. \citet{doumpos2011} provided a comprehensive review of preference disaggregation methods, showing their connections with statistical learning techniques, and highlighting how these approaches can be integrated to enhance decision-support systems. In this context, preference disaggregation emerges as a promising alternative, functioning as an explainable AI that learns and maintains interpretable preferences \citep{Rudin2019}.%---an instance of what is known in the literature as \emph{preference learning}, namely the data-driven inference of an expert's underlying preference model from a limited set of holistic decision examples---aligning with the growing emphasis on explainable models in decision-making .% By ensuring that the learned preferences are easily understandable, preference learning enhances the usability and acceptance of MCDA methods.

Despite these methodological advances, a fundamental challenge remains. Behavioral decision theory suggests that experts often exhibit inconsistencies in their judgment and are influenced by various cognitive biases \citep{TverskyKahneman1974}. While disaggregation methods are powerful and widely applicable, their reliance on a single round of holistic comparisons may leave room for judgment inconsistencies to influence the results. The comparison process might introduce some variability in ranking outcomes, especially when experts' judgments are internally inconsistent, potentially affecting the stability of the inferred value functions \citep{MortonFasolo2009, RezaeiArabMehregan2024, MazurekEtAl2024, sun2025anchor}. This consideration is particularly important for methods like the UTA family, most of which inherently assume a high level of consistency and rationality in the provided preference rankings. %\citep{Vetschera2006,SchillingEtAl2007,KorhonenEtAl2012}.
GRIP (Generalized Regression with Intensities of Preferences) \citep{figueira2009} enriches the UTA family by explicitly modeling preference intensities through a regression framework, thereby offering analysts a more nuanced view of expert judgments. \citet{BarbatiGrecoLami2024} proposed a Deck-of-Cards-based Robust Ordinal Regression (DOR) to further push the use of cardinal preferential information in preference disaggregation methods. More recent work continues to advance these lines, for instance by combining robust ordinal regression with interactive online preference learning~\citep{Grillo2025OrdinalOnline} and by improving consistency and consensus in group BWM~\citep{Lu2026GroupBWMConsensus}.

One practical issue that can arise when applying UTA-type methods is that incomplete preference information may leave the analyst with several (occasionally infinitely many) additive value models that all fit the expert's holistic judgments \citep{GrecoEtAl2008}.
%This information can range from missing rankings to more complex forms like incomplete linguistic preference relations \citep{LiZhang2024}.
This plurality, while faithful to the data, may translate into more than one plausible recommendation. Eliciting additional judgments is an effective way to narrow the solution set, but time or cognitive load constraints sometimes limit how much extra information can be gathered \citep{CiomekEtAl2017}. Robust Ordinal Regression and related robustness-analysis procedures offer valuable insight into the full spectrum of admissible rankings, yet presenting a range of models instead of a single solution can feel abstract to users who prefer an immediately actionable outcome. 

This review reveals a significant research gap. While disaggregation approaches are powerful, they often force a trade-off: methods are either simple but assume judgment consistency or ignore it (like traditional UTA), or they are robust but require specialized expertise and significant time from the expert (like GRIP, DOR, and regularization frameworks). There is a clear need for a method that can \textit{simultaneously} address cognitive biases and internal consistency, while remaining practical and cognitively accessible for the expert.

To further address these challenges, this study introduces a novel solution called Best-Worst Disaggregation (BWD) method, which integrates the idea of BWM within the disaggregation framework of MCDA. Let us note that the idea of combining the two approaches was already suggested by \citet{BarbatiGrecoLami2024} and that their contribution represents the starting point of this manuscript. The BWD method leverages the structured pairwise comparison approach of BWM alongside the foundational principles of disaggregation methods to derive the value functions, reliably and consistently, from the preferences provided by the experts. The proposed method adopts the ``consider-the-opposite strategy'' or counter-anchoring \citep{%MussweilerStrack2000,
MussweilerStrackPfeiffer2000}, by involving comparisons between two opposite reference points (best and worst) and other criteria. This strategy has been incorporated into numerous psychological studies to enhance judgment consistency and demonstrated its efficacy to mitigate the negative effects of the anchoring bias \citep{
Adame2016,
Joslyn2011,
%Mussweiler2002,
Rezaei2021,
%RezaeiEtAl2022,
RezaeiArabMehregan2024,sun2025anchor,sun2025anchor2}. 
From a systems-engineering perspective, BWD is conceived as a human-in-the-loop module of a larger decision-support pipeline: the Best-to-Others and Others-to-Worst comparisons constitute the human input, the linear program acts as the inference engine, and the ordinal and cardinal
consistency ratios close a feedback loop that triggers the revision of unreliable judgments before the value function is finally fitted.

We summarize the novel contributions of this study as follows. To the best of our knowledge, BWD is the first preference disaggregation method that integrates the Best-Worst pairwise-comparison protocol of BWM into the UTA-style ordinal regression framework. Specifically: (i) We introduce a preference disaggregation method that combines BWM with the disaggregation framework in MCDA, in which an additive (possibly piecewise-linear) value function is inferred from holistic Best-to-Others and Others-to-Worst comparisons made on reference alternatives, rather than on criteria as in standard BWM, thus exploiting the cardinal information embedded in Best-Worst judgments using the ``consider-the-opposite'' or counter-anchoring strategy to reduce cognitive bias in the elicitation process, an aspect not addressed by existing disaggregation approaches. (ii) We develop an optimization model to derive the most suitable additive value function based on experts' pairwise comparisons. (iii) We show that BWD can be extended to handle interval-valued preferences to favor its use in uncertain environments. (iv) We propose a systematic approach to select a reference set that covers all criteria levels, minimizes cognitive load, and avoids Pareto dominance. (v) We show that BWD can account for ordinal and cardinal consistency ratios along with sensitivity analysis and extreme ranking analysis. 

Compared with the DOR method~\citep{BarbatiGrecoLami2024}, which first suggested coupling the two paradigms and is the starting point of this manuscript, BWD replaces the cognitively demanding Deck-of-Cards protocol with the cognitively efficient BWM protocol and additionally supports interval-valued judgments, consistency checks, and rank-stability analysis.

The remainder of this paper is structured as follows.
Section \ref{sec:preliminaries} provides preliminary information on the underlying theories of MCDA, focusing on BWM and UTA methods.
Section \ref{sec:BWD} details the proposed BWD method. In its subsections, we explore consistency and sensitivity analysis as well as interval-valued extensions.
Section \ref{sec:case-study} discusses a real-world application of the method.
Finally, Section \ref{sec:conclusions} concludes the paper and suggests directions for future research.

\section{Preliminaries}
\label{sec:preliminaries}

\subsection{Multi-Attribute Value Theory}

MAVT is a framework in MCDA that assumes that each alternative can be characterized by a list of its relevant attributes. A typical MCDA problem consists of a non-empty finite set of $m$ alternatives whose indices range over the set $M=\{1,\ldots,m \}$ and a finite set of $n$ attributes (criteria) indexed by the set $N=\{1,\ldots,n \}$. The set of possible levels that can be achieved by a generic alternative with respect to the $j$th attribute is denoted by $X_j$.
Assuming all relevant attributes have been considered, each alternative can be associated with a consequence vector $\mathbf{x}=(x_1,\ldots,x_n ) \in X_1 \times \cdots \times X_n$, where $x_j \in X_j$ represents the level of the $j$th attribute achieved by the alternative. For simplicity, we will consider consequence vectors $\mathbf{x}$ in place of alternatives.
In MAVT, each attribute can be rescaled and normalized into the interval \([0,1]\) using a marginal value function $v_j:X_j \rightarrow [0,1]$, where we assign values 0 and 1 to the least \(\underline{x}_j \) and most \(\overline {x}_j \) desirable attribute levels, respectively. Hence, $v_j (x_j ) \in [0,1]$ can be interpreted as the value of the consequence evaluated with respect to the $j$th attribute.
The primary challenge in MAVT is to find a function $V:[0,1]^n \rightarrow [0,1]$ that can correctly aggregate the $n$ values $v_j (x_j )$ and represent the preferences of an expert. Namely, given a set of alternatives, each of which can be uniquely represented by an associated consequence vector, the value function $V$ should be coherent with the preorder structure on the set of alternatives:
\begin{equation}
\label{eq:order}
V(\mathbf{x}) \geq V(\mathbf{y})   \Leftrightarrow \mathbf{x}  \succsim \mathbf{y}  \;\; \text{($\mathbf{x}$ is weakly preferred to $\mathbf{y}$)} \, .
\end{equation}

Several theoretical results from the literature help restrict the set of compatible value functions. If the conditions of preference independence and difference independence hold \citep{SmithDyer2021}, there exist attribute-specific value functions $v_j:X_j \rightarrow [0,1]$ such that:
\begin{equation}
\label{eq:additive}
V(\mathbf{x})= \sum_{j=1}^{n} w_j v_j (x_j ) \, ,
\end{equation}
where the scaling constants (also called weights) are such that $w_j \geq 0$ for all $j$, and $w_1+ \cdots + w_n =1$. This type of value function is called \emph{additive value function}. Assuming this formulation, the main task of MAVT methods is to find attribute-specific value functions $v_j$ and scaling constants $w_j$. This is where we can roughly divide MAVT methods into two categories.

The first category contains the methods which use available information on the order relation $\succsim$ to deduce attribute-specific value functions $v_j$ and weights $w_j$. That is, these methods use partial information on the right-hand side of \eqref{eq:order} to deduce its left-hand side, and then apply the elicited left-hand side to find the complete order relation on the set of alternatives. Such methods are called \emph{preference disaggregation} or \emph{preference learning} methods. Among them there are the well-known UTA methods \citep{SiskosEtAl2016}.

The second family of methods uses elicitation techniques on experts to find compatible attribute-specific value functions $v_{j}$'s and scaling constants $w_j$'s. These methods work entirely on the left-hand side of \eqref{eq:order}. Here we have combinations of methods like the mid-value splitting technique and the trade-off method \citep{EisenfuhrEtAl2010}, %the SMART \citep{EdwardsBarron1994},
Swing families \citep{WinterfeldtEdwards1986} and Best-Worst Tradeoff (BWT) \citep{liang2022}. 

%The goal of this study is to propose a new disaggregation method based on the philosophy of the BWM. Here we first present the BWM and then the basic UTA methods as the main existing disaggregation methods. Building up on the main concepts of the two methods we develop the BWD method, which is presented in the next section.

\subsection{Best-Worst Method}

BWM is a pairwise comparison-based method used to elicit criteria weights. The pairwise comparison data is collected from the expert who is asked to first identify a set of relevant decision criteria. Then, the expert should identify the ‘best’ and the ‘worst’ criteria. The ‘best’ criterion is the one which, according to the expert, has the largest contribution to the goal of the decision-making problem, while the ‘worst’ has the smallest contribution. Then, the expert compares the best criterion, $B$, to all the others, and also all the criteria to the worst one $W$, using a number from 1 to 9 (or possibly some other scale). Finally, \citet{Rezaei2015} formulated the following optimization problem to find the weights,
\begin{equation}
\label{eq:BWM}
\begin{aligned}
& \underset{\xi,w_{1},\ldots,w_{n}}{\text{minimize}} & &  \xi \\
& \text{subject to} & & \left| \frac{w_B}{w_j} -a_{Bj} \right| \leq \xi \hspace{0.5cm} \forall j \in N \setminus \{ B \} \, , \\
&  & & \left| \frac{w_j}{w_W} -a_{jW} \right| \leq \xi \hspace{0.5cm} \forall  j \in N \setminus \{ W \} \, , \\
&  & & w_1 + \cdots + w_n = 1 \, , \\
&  & & w_j \geq 0 \hspace{0.5cm} \forall j \in N \, .
\end{aligned}
\end{equation}
where $a_{Bj}$ represents the value of the comparison between the best criterion $B$ and the $j$th criterion, and similarly $a_{jW}$ is the comparison between the $j$th criterion and the worst one, $W$. Optimization problem \eqref{eq:BWM} is non-linear and the optimal value of the objective function can be interpreted as an absolute quantification of inconsistency of the preferences expressed by the expert. There exists a linear version of the BWM, which is presented as follows \citep{Rezaei2016}.
\begin{equation}
\label{eq:BWMlinear}
\begin{aligned}
& \underset{\xi^{L},w_{1},\ldots,w_{n}}{\text{minimize}} & &  \xi^{L} \\
& \text{subject to} & & | w_B-a_{Bj} w_j | \leq \xi^{L} \hspace{0.5cm} \forall j \in N \setminus \{ B \} \, , \\
&  & & |w_j-a_{jW} w_W | \leq \xi^{L} \hspace{0.5cm} \forall  j \in N \setminus \{ W \} \, , \\
&  & & w_1 + \cdots + w_n = 1 \, , \\
&  & & w_j \geq 0 \hspace{0.5cm} \forall j \in N\, .
\end{aligned}
\end{equation}
Similar to the non-linear model, the optimal objective value $\xi^{L}$ serves as an absolute measure of deviation from full consistency.

In BWM, questions on pairs of criteria regard their relative importance in the decision process. Conversely, in MAVT, ratios between weights are interpreted as rates of substitution. To fill this gap, the Best-Worst Tradeoff (BWT) was introduced \citep{liang2022} to adapt the BWM to the context of MAVT. Whereas the interpretations of the values of the comparisons used in BWM and BWT are different, the formulations of the models are sufficiently similar so that the former is still representative of the latter.

\subsection{Preference disaggregation analysis framework}

Preference disaggregation is an approach within MCDA that focuses on deriving criteria weights and value functions from some holistic preferences of experts on a reference set of alternatives \citep{Zionts1981,DoumposEtAl2022}. This approach is particularly valuable when: (i) experts find it more practical to provide overall evaluations of alternatives rather than detailed information on criteria%An example is the application of preference learning to the selection of validators in a blockchain environment \citep{GehrleinEtAl2023}.
; (ii) there is a sufficiently large set of well-known alternatives from which an expert can easily extract a reference set on which she feels sufficiently confident to express her preferences%An example is the study by \citet{CorrenteEtAl2013} on the levels of innovation of European countries.
; (iii) comparing some pairs of criteria is difficult. 
The UTA method, proposed by \citet{JacquetSiskos1982}, is a regression-based approach within the field of MCDA. The goal of UTASTAR, one of its variants, is to derive the additive value functions that best reflect the expert's holistic preferences. The expert provides a ranking on a reference set. Next, a linear optimization problem is solved to derive the attribute-specific value functions for each criterion, ensuring that the attribute-specific value functions replicate the provided rankings as closely as possible. Formally, the steps of UTASTAR are the following.

\paragraph*{Concept and framework}

In a typical MCDA problem, a set of $m$ alternatives $\{ \mathbf{x}_{1}\ldots,\mathbf{x}_{m}\}$ has to be evaluated with respect to $n$ criteria. The UTA method seeks to derive a global value function. To simplify the search for such a function, and knowing that $V$ is unique up to an affine transformation, UTASTAR considers the weights implicitly in the global value function. That is, attribute-specific value functions are normalized in such a way that 
\begin{equation}
\label{eq:1c}
\sum_{j=1}^{n} v_j (\overline{x}_j )=1.
\end{equation}
By doing so, explicit reference to weights can be omitted, as they simply correspond to the values $v_j (\overline{x}_j)$ and the value of an alternative $\mathbf{x}$ becomes
$V(\mathbf{x}) = \sum_{j=1}^{n} v_j (x_j) $.

\paragraph*{Linear Programming and Inference}

\citet{JacquetSiskos1982} proposed to use piecewise linear approximations of the attribute-specific value functions $v_j$. For the $j$th criterion, the interval $ [\underline{x}_j, \overline{x}_j ]$ can be divided into $(s_{j} -1)$ (with $s_j \in \mathbb{N}_{\geq 1}$) subintervals with points $x_{j}^{0}, ..., x_{j}^{s_j}$, where $x_{j}^{0}=\underline{x}_{j}$ and $x_{j}^{s_j}=\overline{x}_j$, and the approximate attribute-specific value function $v_j$ can be defined as follows. If $x_j \in [ x_j^k,x_j^{k+1} ]$, then
\begin{equation}
v_j (x_j)= v_j(x_j^k ) + \frac{ x_j-x_j^k}{x_j^{k+1}-x_j^k } \left[v_j (x_j^{k+1} )-v_j (x_j^k )\right] \, .
\end{equation}
Assuming monotonicity, each piecewise linear attribute-specific value function $v_j (x_j )$ satisfies the following constraints ($j$ is a benefit criterion):
\begin{equation}
\label{eq:monotonicity}
 v_j (x_j^k ) \leq v_j (x_j^{(k+1)} )  \;\; \forall j \text{ and }  \forall k.
\end{equation}
The value function resulting from the aggregation of the piecewise linear attribute-specific value functions is, then,
\begin{equation}
\label{eq:piecewiseV}
V(\mathbf{x})= \sum_{j=1}^{n} v_j (x_j)  \, .
\end{equation}
At this point, the UTASTAR method assumes that an expert can provide a preference relation $\succsim$ on a selected reference index-set of alternatives $R$,\footnote{Note that the reference set of alternatives may be a subset of the decision alternatives, i.e., $R \subseteq M$. Nevertheless, this is not necessary as the reference set of alternatives may contain other alternatives.} and tries to find the piecewise linear value function $V$ that best fits the ranking. That is, if we consider two alternatives $\mathbf{x}$ and $\mathbf{y}$ in the reference set, then
\begin{equation}
\label{eq:prefR}
    \begin{cases}
        V(\mathbf{x})-V(\mathbf{y})  \geq \delta &  \text{if }\mathbf{x}  \succ \mathbf{y} \\
        V(\mathbf{x})-V(\mathbf{y})  = 0         &  \text{if }\mathbf{x}  \sim \mathbf{y}
    \end{cases} \, ,
\end{equation}
where $\delta$ is a small positive number and $\succ$ and $\sim$ are the asymmetric and symmetric parts of $\succsim$, respectively. Then, when the attribute-specific value functions are piecewise linear, the true value of the generic alternative $\mathbf{x}$ can be expressed as,
\begin{equation}
\label{eq:approx}
\tilde{V}(\mathbf{x})= \sum_{j=1}^{n} v_j (x_j) -\sigma_{\mathbf{x}}^+ +\sigma_{\mathbf{x}}^- \, ,
\end{equation}
where $x_j$ is the attribute level achieved by alternative $\mathbf{x}$ with respect to the $j$th attribute and $\sigma_{\mathbf{x}}^+$ and $\sigma_{\mathbf{x}}^-$ are overestimation and underestimation errors, respectively. Albeit not a strict requirement, pragmatism suggests that the cardinality of $R$ be significantly smaller than the cardinality of $M$.

 \paragraph*{Optimization}

 To obtain good approximations for the piecewise linear attribute-specific value functions, one accounts for conditions \eqref{eq:1c}--\eqref{eq:approx}, and solves the following problem.
\begin{equation}
\label{eq:UTASTAR}
\begin{aligned}
\theta = \; & \text{minimize} & & \sum_{i \in R} \left( \sigma_{\mathbf{x}_i}^+ +\sigma_{\mathbf{x}_i}^- \right) \\
& \text{subject to} & &  \tilde{V}(\mathbf{x}_{p})-\tilde{V}(\mathbf{x}_{q})  \geq \delta   \text{ if }\mathbf{x}_{p}  \succ \mathbf{x}_{q} \, , \\
& & &    \tilde{V}(\mathbf{x}_{p})-\tilde{V}(\mathbf{x}_{q})  = 0     \text{ if }\mathbf{x}_{p}  \sim \mathbf{x}_{q} \, , \\
&  & &  v_j (x_j^k ) \leq v_j (x_j^{k+1} )   \hspace{0.5cm} \forall j \text{ and }  \forall k \, , \\
&  & & \sum_{j=1}^{n} v_j (\overline{x}_j )=1 \, , \\
& & & v_j (\underline{x}_j )=0  \hspace{0.5cm} \forall j \in N \, , \\
& & & \sigma_{\mathbf{x}_i}^+ , \sigma_{\mathbf{x}_i}^- \geq 0 \hspace{0.5cm} \forall i \in R \, .
\end{aligned}
\end{equation}
If the optimal value of the objective function is $\theta=0$, then the constructed attribute-specific value functions can reflect the expert's preferences. Conversely, if $\theta>0$, then the attribute-specific value functions fail to reflect them and thus are only approximations.

\paragraph*{Stability} Given the possible existence of multiple optimal or near-optimal solutions, it is suggested to take the average of the additive value functions that maximize and minimize the objective function $v_j(\overline{x}_j)$ for all $j \in N$, with the same constraints as in \eqref{eq:UTASTAR} and the addition of 
\[
\sum_{i \in R} \left( \sigma_{\mathbf{x}_i}^+ +\sigma_{\mathbf{x}_i}^- \right) \leq \theta^{*}+\varepsilon \, ,
\] 
where $\theta^{*}$ is the optimum of \eqref{eq:UTASTAR} and $\varepsilon >0$ is extremely small. This allows, by alternating the maximization and the minimization of all the $v_j(\overline{x}_j)$, to find a finite number of extreme solutions within the set of near-optimal solutions. As the set of near-optimal solutions defined by this new optimization problem is a convex set (a polytope), their average must be feasible and near-optimal too.  

%In this case, the following methods can be attempted to obtain a attribute value function compatible with the decision-maker's preference information: (i) increase the number of segments for one or more criteria (ii) adjust or reselect the reference set $R$ (iii) relax the constraints of the feasible region, i.e., search for feasible solutions within the range $\sigma \leq \sigma^*+ \varepsilon.

%Over the years, several variants and extensions of UTA methods have been developed to enhance its applicability and robustness. For instance, the UTADIS variant caters to sorting problems \citep{doumpos2002}. Another notable extension is the Quasi-UTA model, which simplifies the process by using recursive exponential functions to reduce the information needed to build the utility function \citep{BeutheEtAl2000}.
Since its inception, UTA methods have found applications across various fields \citep{DiasEtAl2018}, including commercial strategy, venture capital investment evaluation, country risk assessment, business financing, and portfolio management \citep{SiskosEtAl2016}. Methods such as UTA\textsuperscript{GMS} and GRIP incorporate robust ordinal regression to handle interacting criteria and enhance the decision-making process \citep{GrecoEtAl2008}. These advancements ensure that the UTA methods remain relevant and powerful tools in the field of MCDA \citep{DoumposEtAl2022}.

\section{Best-worst disaggregation method}
\label{sec:BWD}

Suppose we have a set of $m$ alternatives, represented by the profile vectors $\mathbf{x}_1,\ldots,\mathbf{x}_m$, which are evaluated with respect to a set of $n$ criteria. The performance matrix can be represented as follows,
\[
\mathbf{X}=
\begin{pmatrix}
\mathbf{x}_1 \\
\vdots \\
\mathbf{x}_m
\end{pmatrix}
=
\begin{pmatrix}
x_{11} & \cdots & x_{1n} \\
\vdots & \ddots & \vdots \\
x_{m1} & \cdots & x_{mn}
\end{pmatrix},
\]
where the $i$th row is the profile of the $i$th alternative, i.e., $\mathbf{x}_i$. Following the precepts of UTASTAR, we identify a reference set of alternatives and we call $R$ its index set. Then, according to the principles of BWM, we identify the indices of the best and the worst alternatives in the reference set and we label them with the letters $B$ and $W$, respectively. We note in passing that the freedom in the definition of the reference set $R$ allows the expert to choose it in such a way as to simplify the choice of the best and the worst alternative within it. We then conduct a holistic comparison between the best alternative $\mathbf{x}_B$ and the other alternatives, as well as between all the alternatives and the worst alternative $\mathbf{x}_W$. Such comparisons can be collected into two sets as follows:
\begin{equation*}
A^{BO}=\left\{ a_{Bi }  \,|\, i \in R \right\} \quad \quad
A^{OW}=\left\{ a_{iW }  \,|\, i \in R \right\}.
\end{equation*}
Set $A^{BO}$ is called the ‘Best-to-Others’ set and its element $a_{Bi}$ shows the expert’s preference of the ‘best alternative’ $\mathbf{x}_B$ over alternative $\mathbf{x}_i$. Similarly, set $A^{OW}$ is called the ‘Others-to-Worst’ set and its generic element $a_{iW }$ represents the expert's preference of alternative $\mathbf{x}_i$ over the worst alternative $\mathbf{x}_W$. 

Preferences are expressed on the standard BWM 1--9 ratio scale, where $a_{Bi}, a_{iW} \in \{1,\ldots,9\}$. Following \citet{Rezaei2015}, $a_{Bi}=k$ means that alternative $\mathbf{x}_B$ is $k$ times more preferred than alternative $\mathbf{x}_i$, implying $V(\mathbf{x}_B)/V(\mathbf{x}_i)=k$. The scale follows the usual linguistic interpretation, from equally preferred to extremely more preferred. This interpretation ensures that the elicited preferences $a_{Bi}$ and $a_{iW}$ directly correspond to the value-function ratios in optimization model \ref{eq:BWD}. Alternative scales and customization methods are also available \citep{DuttaEtAl2025}.

In a situation expressing full rationality, $V(\mathbf{x}_B )=a_{B i } V(\mathbf{x}_i )$ and $V(\mathbf{x}_i )=a_{i W } V(\mathbf{x}_W )$ for all $i \in R$. Then, 
\begin{equation}
\begin{cases}
\label{eq:ideal}
V(\mathbf{x}_B )- a_{B i } V(\mathbf{x}_i ) =0 \\
V(\mathbf{x}_i )- a_{i W } V(\mathbf{x}_W ) =0
\end{cases} \hspace{3mm} \forall i \in R \, .
\end{equation}
As such, the objective would be to minimize the maximum deviation from the ideal situation \eqref{eq:ideal}, which can be formulated as the following linear optimization problem:
\begin{equation}
\label{eq:BWD}
\tag{BWD}
\begin{aligned}
\xi^{*} = \;  &\text{min.} & & 
\xi \\
& \text{s.t.} & & - \xi \leq V(\mathbf{x}_B ) - a_{B i} V(\mathbf{x}_i) \leq \xi \hspace{0.5cm}  \forall {i} \in R \setminus \{B \} \, , \\
& & &    -\xi \leq V(\mathbf{x}_i ) - a_{i W} V(\mathbf{x}_W ) \leq \xi \hspace{0.3cm}  \forall {i} \in R \setminus \{ W \} \, , \\
&  & &  v_j (x_j^k ) \leq v_j (x_j^{k+1} ) \hspace{0.5cm} \forall j \in N \text{ and }  \forall k \, , \\
&  & & \sum_{j=1}^{n} v_j (\overline{x}_j )=1 \, , \\
& & & v_j (\underline{x}_j )=0  \hspace{0.3cm} \forall j \in N \, .
\end{aligned}
\end{equation}

Solving \ref{eq:BWD}, we can obtain the piecewise linear approximations of the attribute-specific value functions, and its optimum, $\xi^{*}$, is an indicator of  \emph{compatibility} of the preferences of the expert with the most suitable underlying value function $V$. If the results are deemed reliable, then the linear approximations can be used to rate the entire set of alternatives $\{ \mathbf{x}_1,\ldots,\mathbf{x}_m \}$. Coherently with the literature on preference disaggregation, we suggest that further analyses, especially sensitivity, be carried out to test the robustness of the results. Note that we modeled the BWD compatibly with the linear version of the BWM for its simplicity, as this ensures \ref{eq:BWD} remains a linear program. Other versions of the BWM can be used as well, though they may result in a non-linear model.

We note that \citet{BarbatiGrecoLami2024} already suggested a formulation compatible with \ref{eq:BWD} (see eq. (14) in their paper), but with differences: (i) they employ the Deck-of-Cards method to rate alternatives, and (ii) they consider the final rating obtained from the Deck-of-Cards method instead of a set of comparisons $A^{BO} \cup A^{OW}$. Hereafter, we will focus on \ref{eq:BWD} and propose methods to study consistency, compatibility and extreme rankings. Furthermore, we stress that we employed a min-max approach to remain coherent with the original BWM but, as shown also by \citet{BarbatiGrecoLami2024}, the optimization problem can be easily modified to accommodate other formulations, e.g., the sum of the errors.

\subsection{A proposal to select the reference set}

The choice of the reference set $R$ can be seen as a preprocessing phase, but this should not diminish its importance. While in existing disaggregation studies, the reference set is usually given, here, we consider some desiderata to propose a method to select the reference set:
\begin{enumerate}[start=1,label={(D\arabic*):}]
    \item For all the attributes, the selected alternatives should cover all the segments into which the attribute ranges were divided. This ensures that there is preferential information on each segment of each attribute.
    \item Even considering the first desideratum, the reference set should be as informative as possible. Let us provide an example to justify our approach. If we consider two alternatives $a$ and $b$ such that $a$ outperforms $b$ with respect to all attributes, then knowing the degree of preference of $a$ over $b$ does not unveil information on the relative importance of attributes. Consider, instead, the two alternatives $c$ and $d$, and only the two attributes $j$ and $j'$ such that $c$ scores better than $d$ with respect to $j$ and, vice versa, $d$ is better than $c$ with respect to $j'$. In this case, knowing that $d$ is considered better than $c$ may give information on the greater importance of attribute $j'$ over attribute $j$. Note also that $a$ Pareto dominates $b$, and that no Pareto dominance exists between $c$ and $d$. So, it may be wise to prefer a reference set where no contained alternative is dominated. Namely, we want to avoid stratified reference sets. 
    \item The cardinality of the reference set should be minimized to decrease the cognitive load. This is coherent with current parsimonious approaches to the elicitation of preferences \citep{AbastanteEtAl2019,corrente2024better}.
\end{enumerate}

It is necessary to define auxiliary parameters and sets to formulate the optimization problem that puts together the three desiderata. We recall that each criterion is partitioned into $s_{j}$ (possibly equally wide) subintervals. For simplicity, we assume that $s_{j}$ is the same for each criterion and we call it $s$. At this point, one can simply construct a binary array $\mathbf{A}=(a_{ijk})_{m \times n \times s} \in \{ 0,1\}^{m \times n \times s} $
where
\[
a_{ijk}=
\begin{cases}
1, & \begin{aligned} &\text{if the $i$th alternative covers the $k$th segment of the $j$th attribute} \end{aligned} \\
0, & \text{otherwise.}
\end{cases}
\]
We define $\mathcal{D}$ as the set of pairs $(i,i')$ with $i < i'$ and for which one alternative Pareto dominates another.

We propose a combinatorial optimization problem such that each segment is covered by at least $b \in \mathbb{N}_{+}$ alternatives and that in the selected set there is not any Pareto dominance relation between alternatives. That is, we can solve
\begin{alignat}{3}
%\label{eq:se}
 & {\text{minimize}} \quad & & \sum_{i \in M} y_{i} \label{eq:OF} \\
 & \text{subject to} \quad  & &  \sum _{i \in M} a_{ijk} y_i  \geq b \hspace{0.5cm} \forall j,k \label{eq:c1} \, , \\
 & & & y_{i}+y_{i'} \leq 1 \hspace{0.5cm} \forall (i,i') \in \mathcal{D} \label{eq:c2} \, , \\
 & & &    y_{i} \in \{ 0,1\} \hspace{0.5cm} \forall i \in M \label{eq:c3} \, .
\end{alignat}
The objective function \eqref{eq:OF} is a simple count of the number of alternatives in the reference set, and its minimization can be associated with (D3). Constraints \eqref{eq:c1} imply that each attribute segment should be covered by at least $b$ alternatives and thus they are expressive of (D1). With $b=1$ the optimization problem minimizes the cardinality of the reference set such that each attribute segment is covered at least once.
Constraints \eqref{eq:c2} entail that if there is a Pareto dominance relation between alternatives $i$ and $i'$, then at most one of them can be in the reference set. Hence, the optimal reference set cannot contain any two alternatives such that one Pareto dominates the other, and \eqref{eq:c2} corresponds to (D2). Finally, $\eqref{eq:c3}$ simply states the variable domain.
Furthermore, an expert may not have sufficient expertise with respect to some alternatives. In this case, it is simple to modify the optimization problem to impose that such alternatives cannot be selected. A pragmatic approach could be imposing $y_{i}=0$ for all such alternatives. {It is also possible to set some variables $y_i$ equal to 0, to explore the set of optimal solutions, which need not be a singleton.}

%It is also worth mentioning that the purpose of the proposed approach here is to find a non-dominated set as a reference point, not to find the Pareto optimal set of the alternative set. That is, the remaining alternatives which are not selected as the reference set are not necessarily dominated alternatives.  Lastly, a cautionary note: it may be increasingly difficult to solve this problem for a growing number of splitting points. In fact, the greater number of segments to be covered conflicts with the constraint that the chosen reference alternatives cannot dominate each other. 

\subsection{Necessary relation from the BWD model}

It could happen that the optimization problem \ref{eq:BWD} yields infinitely many optimal solutions. For this reason, we suggest an approach that tries to find the necessary ranking of the alternatives. That is, we propose an approach inspired by Robust Ordinal Regression \citep{CorrenteEtAl2013}. First, we call $\xi^{*}$ the optimal value of \ref{eq:BWD}. For each pair of alternatives $(\mathbf{x}_{p},  \mathbf{x}_{q})$, we maximize their value difference as follows:
\begin{equation}
\label{eq:BWDnec}
\begin{aligned}
\delta_{pq} = \; &  \text{max.} & & V(\mathbf{x}_{p}) - V(\mathbf{x}_{q}) \\
& \text{s.t.} & & - \xi^{*} \leq V(\mathbf{x}_B ) - a_{B i} V(\mathbf{x}_i) \leq \xi^{*}\hspace{0.1cm}  \forall {i} \in R \setminus \{B \} \, , \\
& & &    -\xi^{*} \leq V(\mathbf{x}_i ) - a_{i W} V(\mathbf{x}_W ) \leq \xi^{*} \hspace{0.1cm}  \forall {i} \in R \setminus \{ W \} \, , \\
&  & &  v_j (x_j^k ) \leq v_j (x_j^{k+1} ) \hspace{0.5cm} \forall j \in N \text{ and }  \forall k \, , \\
&  & & \sum_{j=1}^{n} v_j (\overline{x}_j )=1 \, , \\
& & & v_j (\underline{x}_j )=0  \hspace{0.5cm} \forall j \in N \, .
\end{aligned}
\end{equation}

Clearly, if $\delta_{pq}<0$, then for all optimal value functions, $V(\mathbf{x}_p) - V(\mathbf{x}_q) < 0$. This implies that $\mathbf{x}_q$ is always valued higher than $\mathbf{x}_p$ across all optimal solutions. Hence, we can say that $\mathbf{x}_q$ is necessarily preferred to $\mathbf{x}_p$. In shorthand notation, we denote this necessary preference as $\mathbf{x}_q \succ_{\mathcal{N}} \mathbf{x}_p$. By repeatedly solving \eqref{eq:BWDnec} for all pairs of distinct indices $({p},{q}) \in M \times M$, we can elicit the necessary preference relation $\succ_{\mathcal{N}}$ across the entire set of alternatives. The results can be represented by a Hasse diagram.

\subsection{Consistency analysis}
%\label{sec:consistency}

In the previous section, we recognized that the optimal value $\xi^{*}$ contains information on the discrepancy between the preferences expressed by the expert and the most suitable underlying model $V$. Hence, we can interpret $\xi^{*}$ as an index of \emph{compatibility}. It is worth noting that compatibility should not be confused with the concept of (internal) \emph{consistency}, which, in the case of the BWM, collapses into the condition $a_{Bj}a_{jW}=a_{BW}$ for all $j \in N$. Consider, for instance, that the preferences of the expert can be perfectly consistent, i.e., rational, and yet violate the principle of Pareto dominance between alternatives at the same time. This latter violation implies that the preferences cannot be fully compatible with the underlying model.

The conceptual difference between compatibility and consistency is shown in Fig.~\ref{fig:comp_cons}. Consistency is an intrinsic property of the preferences expressed by an expert, while compatibility depends on both the preferences and the underlying model.
\begin{figure}[htb]
    \centering
    \includegraphics[width=0.7\linewidth]{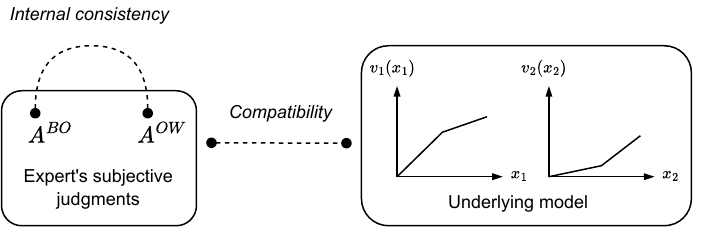}
    \caption{Compatibility and consistency.}
    \label{fig:comp_cons}
\end{figure}
Consequently, it might be more challenging to evaluate the significance of the value of the index $\xi^{*}$, especially since it is debatable whether this value should be minimized. In reality, $\xi^{*}$ is influenced by the design parameters, including the splitting points, of the model in question. An increase in the number of splitting points allows the underlying model to be more flexible and increases its alignment with preferences. However, artificially inflating the number of splitting points to improve compatibility could lead to overfitting issues.

We further note that, within preference disaggregation, the evaluation of internal consistency is a distinctive feature of BWD. Other models, such as UTA and UTASTAR, require as input preference orders that are, by definition, transitive; furthermore, they provide only ordinal information, so any analysis of cardinal consistency is not applicable. We conclude, also in light of Fig.~\ref{fig:comp_cons}, that the analysis of compatibility is not enough and it should be coupled with an analysis of the consistency of the preferences. Starting from the principle that the level of inconsistency of preferences should not be model-dependent, we therefore adopt the Ordinal Consistency Ratio ($OR$) and the Cardinal Consistency Ratio ($CR$) proposed by \citet{liang2020} to quantify deviations from these consistency conditions in the BWD method. $OR$ focuses on the ranking of alternatives in pairwise comparisons, ensuring that the relative ordering of preferences is maintained. It is calculated as
\begin{equation}
\label{eq:or}
\begin{aligned}
OR &= \max_i OR_i \\
&= \max_i \left\{ \frac{1}{n} \sum_{k=1}^{n} F\left((a_{Bk} - a_{Bi}) \cdot (a_{iW} - a_{kW})\right) \right\} \, ,
\end{aligned}
\end{equation}
where 
%$OR_i$ represents the local ordinal consistency ratio for alternative $i$, measuring the proportion of ordinal violations involving that specific alternative. The overall Ordinal Consistency Ratio ($OR$) is the maximum of all local ratios, indicating the worst ordinal inconsistency in the system. The function \(F(x)\) is a step function, which is defined as: }
\[
 F(x) = 
\begin{cases}
1, & \text{if } x < 0, \\
0.5, & \begin{split}  \text{if } x= 0 \text{ and } ( &a_{Bk} - a_{Bi}\neq 0  \text{ or } a_{iW} - a_{kW}\neq 0) \end{split}, \\
0, & \text{otherwise.}
\end{cases}
\]
%The Cardinal Consistency Ratio complements the ordinal consistency analysis by examining the numerical accuracy of pairwise comparisons. The cardinal consistency check evaluates whether the magnitude of preference differences is proportionally consistent across all comparisons, ensuring that the numerical values accurately reflect the decision-maker's intended preference intensities. 
To evaluate cardinal consistency, we use the input-based $CR$,
\begin{equation}
\label{eq:cr}
CR=
\begin{cases}
\max_{i} \dfrac{|a_{Bi}a_{iW}-a_{BW}|}{|a_{BW}a_{BW}-a_{BW}|}, &a_{BW}\geq 2\\
0, &a_{BW}=1 \, .
\end{cases}
\end{equation}
which is equal to $0$ if and only if the preferences are fully consistent and positive otherwise.

Identifying inconsistent judgments and understanding the extent of their deviation from full consistency is only part of the process. It is crucial to determine acceptable thresholds for inconsistency to ensure reliable decision outcomes. Since consistency measurement in BWD mirrors that in BWM, we adopt the thresholds for the input-based consistency ratio proposed by \citet{liang2020}. If the consistency ratios fall below these thresholds, the judgments are considered sufficiently consistent. However, if inconsistency exceeds these thresholds, an iterative process is employed to improve consistency. %This process involves identifying the most inconsistent judgments, providing feedback to the expert, revising the inconsistent comparisons, recalculating the $OR$ and $CR$, and repeating these steps until the judgments fall within the acceptable consistency range \citep{liang2022}. %This approach ensures that the final preferences are refined and more consistent, leading to reliable and defensible decisions. Readers can refer to the original work  \citep{liang2022} for more details.

\subsection{Extreme ranking analysis}

A number of sensitivity analysis tools already proposed for robust ordinal regression can be easily implemented for the BWD model. Two examples are the flexible rankings \citep{KoksalanEtAl2010} and the \emph{extreme ranking analysis} \citep{KadzinskiGrecoSlowinski2012}. In some decision-making contexts, it is often important to determine the maximum and minimum possible ranks that can be achieved by an alternative in the final ranking. This type of sensitivity analysis helps assess the stability of an alternative's position. To determine these bounds, we solve two optimization problems for each alternative, always looking within the set of optimal solutions of \ref{eq:BWD}.

The outranking count ($\overline{d}_{i}$) represents the maximum number of alternatives that the given alternative $\mathbf{x}_{i}$ can outrank at the same time. To determine this value, we solve the following optimization problem: 

\begin{equation}
\label{eq:BWDmaxrank}
\begin{aligned}
\overline{d}_{i} = \; & \text{max.} & & \sum_{h \in M \setminus \{i\} } z_{h} \\
& \text{s.t.} & & - \xi^{*} \leq V(\mathbf{x}_B ) - a_{B i} V(\mathbf{x}_i) \leq \xi^{*}\hspace{0.1cm}  \forall {i} \in R \setminus \{B \} \, , \\
& & &    -\xi^{*} \leq V(\mathbf{x}_i ) - a_{i W} V(\mathbf{x}_W ) \leq \xi^{*} \hspace{0.1cm}  \forall {i} \in R \setminus \{ W \}  \, , \\
& & & V(\mathbf{x}_{i})+ \varepsilon \cdot z_{h} \geq V(\mathbf{x}_{h}) -  q_{h} \hspace{0.5cm} \forall h \in M \setminus \{i\} \, , \\
& & & z_{h} + q_{h} = 1 \hspace{0.5cm} \forall h \in M \, , \\
&  & &  v_j (x_j^k ) \leq v_j (x_j^{k+1} ) \hspace{0.5cm} \forall j \text{ and }  \forall k \, , \\
&  & & \sum_{j=1}^{n} v_j (\overline{x}_j )=1 \, , \\
& & & v_j (\underline{x}_j )=0  \hspace{0.5cm} \forall j \in N \, , \\
& & & z_{h},q_{h} \in \{ 0,1 \} \hspace{0.5cm} \forall h \in M \, .
\end{aligned}
\end{equation}

where, again, $\varepsilon$ is an extremely small positive value. The optimization problem searches in the space of the optimal solutions of \ref{eq:BWD} the one that maximizes the number of alternatives that can simultaneously be judged worse than $\mathbf{x}_{i}$. This is done by maximizing the sum of the binary variables $z_{h}$, where $z_{h}=1$ if and only if, for that given compatible value function, $V(\mathbf{x}_{i}) > V(\mathbf{x}_{h})$. The best possible rank ($\overline{r}_{i}$) of alternative $\mathbf{x}_{i}$ is then $\overline{r}_{i} = m - \overline{d}_{i}$, where $m$ represents the total number of alternatives evaluated in the decision-making problem.

Similarly, the dominance count ($\underline{d}_{i}$) represents the maximum number of alternatives that can simultaneously be strictly preferred to $\mathbf{x}_{i}$. It is sufficient to solve the following problem to find $\underline{d}_{i}$:
\begin{equation}
\label{eq:BWDminrank}
\begin{aligned}
\underline{d}_{i} = \; & \text{max.} & & \sum_{h \in M \setminus \{ i\} } z_{h} \\
& \text{s.t.} & & - \xi^{*} \leq V(\mathbf{x}_B ) - a_{B i} V(\mathbf{x}_i) \leq \xi^{*}\hspace{0.1cm}  \forall {i} \in R \setminus \{B \} \, , \\
& & &    -\xi^{*} \leq V(\mathbf{x}_i ) - a_{i W} V(\mathbf{x}_W ) \leq \xi^{*} \hspace{0.1cm}  \forall {i} \in R \setminus \{ W \} \, , \\
& & & V(\mathbf{x}_{h})+ \varepsilon \cdot z_{h} \geq V(\mathbf{x}_{i}) -  q_{h} \hspace{0.5cm} \forall h \in M \setminus \{i\} \, , \\
& & & z_{h} + q_{h} = 1 \hspace{0.5cm} \forall h \in M \, , \\
&  & &  v_j (x_j^k ) \leq v_j (x_j^{k+1} ) \hspace{0.5cm} \forall j \text{ and }  \forall k \, , \\
&  & & \sum_{j=1}^{n} v_j (\overline{x}_j )=1 \, , \\
& & & v_j (\underline{x}_j )=0  \hspace{0.5cm} \forall j \in N \, , \\
& & & z_{h},q_{h} \in \{ 0,1 \} \hspace{0.5cm} \forall h \in M \, .
\end{aligned}
\end{equation}

The worst possible rank ($\underline{r}_{i}$) of alternative $\mathbf{x}_{i}$ is then calculated as $\underline{r}_{i} = \underline{d}_{i} + 1$. By definition, $\overline{r}_{i} \leq \underline{r}_{i}$, and therefore we define the \emph{range of possible ranks} as $[\overline{r}_{i}, \underline{r}_{i}]$ for each alternative $i \in M$ with the goal of providing valuable insight into the stability and variability of the alternative's ranking.

The whole BWD pipeline is summarized in Fig.~\ref{fig:BWD_flowchart}.

\begin{figure}[ht]
\centering

\resizebox{0.5\columnwidth}{!}{
\begin{tikzpicture}[
    font=\footnotesize,
    node distance=0.55cm,
    startstop/.style={rectangle, rounded corners=4pt, draw=black, thick,
        minimum width=1.8cm, minimum height=0.55cm, align=center},
    process/.style={rectangle, draw=black, thick,
        minimum width=5.2cm, minimum height=0.6cm, align=center},
    decision/.style={diamond, draw=black, thick, aspect=2.6,
        inner sep=1pt, align=center},
    arrow/.style={-{Latex[length=2mm]}, thick}
]
\node (start)  [startstop]                       {Start};
\node (input)  [process, below=of start]         {Input: decision matrix $\mathbf{X}$ and criteria $N$};
\node (ref)    [process, below=of input]         {Select reference set $R$ via set-covering};
\node (bw)     [process, below=of ref]           {Identify best $\mathbf{x}_B$ and worst $\mathbf{x}_W$ in $R$};
\node (elicit) [process, below=of bw]            {Elicit $A^{BO}$ and $A^{OW}$ (crisp or interval)};
\node (lp)     [process, below=of elicit]        {Solve BWD linear program~\eqref{eq:BWD} $\Rightarrow\;v_j,\,\xi^{*}$};
\node (cons)   [process, below=of lp]            {Compute consistency ratios $OR$ and $CR$};
\node (dec)    [decision, below=0.6cm of cons]   {$OR$ and $CR$ below thresholds?};
\node (out)    [process, below=0.6cm of dec]     {Compute $V$, necessary relation $\succ_{\mathcal{N}}$, and extreme ranks};
\node (end)    [startstop, below=of out]         {End};
\node (feed)   [process, right=0.9cm of dec, minimum width=3.3cm]
               {Feedback: revise \\ inconsistent judgments};

\draw [arrow] (start)  -- (input);
\draw [arrow] (input)  -- (ref);
\draw [arrow] (ref)    -- (bw);
\draw [arrow] (bw)     -- (elicit);
\draw [arrow] (elicit) -- (lp);
\draw [arrow] (lp)     -- (cons);
\draw [arrow] (cons)   -- (dec);
\draw [arrow] (dec)    -- node[right, font=\scriptsize] {Yes} (out);
\draw [arrow] (dec)    -- node[above, font=\scriptsize] {No}  (feed);
\draw [arrow] (feed.north) |- (elicit.east);
\draw [arrow] (out)    -- (end);
\end{tikzpicture}
}
\caption{Flowchart of the BWD methodology.}
\label{fig:BWD_flowchart}

\end{figure}

\subsection{An interval-valued extension}
%\label{sec:interval}

Many decision analysis methods have been extended to accommodate uncertain and imprecise preference information. This need is particularly significant in the BWD method. Here, cardinal information is elicited for holistic comparisons between alternatives, which are often more cognitively demanding than specific comparisons between criteria levels or in the form of trade-offs. Additionally, the interval-based approach allows experts to express uncertainty and imprecision more naturally by providing bounds on their judgments rather than a single precise value; moreover, no assumption is made on the underlying probability distribution \citep{SaloHamalainen1995}.

% Using this approach, when an expert compares two alternatives $\mathbf{x}_i$ and $\mathbf{x}_j$, then we can assume that her judgement $\tilde{a}_{ij}=[{a}^{-}_{ij},{a}^{+}_{ij}]$ is an interval, meaning that, $V(\mathbf{x}_{i})/V(\mathbf{x}_{j}) \in [{a}^{-}_{ij},{a}^{+}_{ij}]$. Equivalently,
% \begin{equation}
% \label{eq:intervals_tight}
% a_{ij}^{-} V(\mathbf{x}_{j}) \leq V(\mathbf{x}_{i}) \leq a_{ij}^{+} V(\mathbf{x}_{j}) \, .
% \end{equation}
% Nevertheless, it is possible that, due to conflicting preferences and inconsistencies, there may not exist an additive value function $V$ satisfying \eqref{eq:intervals_tight} and therefore one may want to relax its formulation using the slack variable $\xi$ as follows:
% \begin{equation}
% \begin{aligned}
% \label{eq:intervals_relax}
% -\xi + a_{ij}^{-} V(\mathbf{x}_{j}) \leq V(\mathbf{x}_{i}) \leq a_{ij}^{+} V(\mathbf{x}_{j}) + \xi \, .
% \end{aligned}
% \end{equation}
%

Using this approach, when an expert compares reference alternatives, we assume that her judgments can be intervals, i.e., 
\begin{equation}
\begin{aligned}
\tilde{a}_{Bi}&=[{a}^{-}_{Bi},{a}^{+}_{Bi}]\hspace{0.5cm} \forall {i} \in R \setminus \{B \} \, , \\
\tilde{a}_{iW}&=[{a}^{-}_{iW},{a}^{+}_{iW}] \hspace{0.5cm} \forall {i} \in R \setminus \{ W \} \, ,
\end{aligned}
\end{equation}
with the interpretations
$[{a}^{-}_{Bi},{a}^{+}_{Bi}] \ni V(\mathbf{x}_{B})/V(\mathbf{x}_{i})$ and $[{a}^{-}_{iW},{a}^{+}_{iW}] \ni V(\mathbf{x}_{i})/V(\mathbf{x}_{W})$. Equivalently,
\begin{equation}
\begin{aligned}
\label{eq:intervals_tight}
& a_{Bi}^{-} V(\mathbf{x}_{i}) \leq V(\mathbf{x}_{B}) \leq a_{Bi}^{+} V(\mathbf{x}_{i}) \hspace{0.4cm} \forall {i} \in R \setminus \{ B \} \, , \\
& a_{iW}^{-} V(\mathbf{x}_{W}) \leq V(\mathbf{x}_{i}) \leq a_{iW}^{+} V(\mathbf{x}_{W}) \hspace{0.4cm} \forall {i} \in R \setminus \{ W \} \, .
\end{aligned}
\end{equation}

Nevertheless, it is possible that, due to conflicting preferences and inconsistencies, there may not exist an additive value function $V$ satisfying \eqref{eq:intervals_tight} and therefore one may want to relax its formulation using the slack variable $\xi$ as follows:
\begin{equation}
\begin{aligned}
\label{eq:intervals_relax}
& -\xi + a_{Bi}^{-} V(\mathbf{x}_{i}) \leq V(\mathbf{x}_{B}) \leq a_{Bi}^{+} V(\mathbf{x}_{i}) + \xi  \hspace{0.2cm} \forall {i} \in R \setminus \{ B \} \, , \\
& -\xi + a_{iW}^{-} V(\mathbf{x}_{W}) \leq V(\mathbf{x}_{i}) \leq a_{iW}^{+} V(\mathbf{x}_{W}) + \xi \hspace{0.2cm} \forall {i} \in R \setminus \{ W \} \, . 
\end{aligned}
\end{equation}

Consequently, we propose the following linear optimization problem as an extension of the BWD that can consider interval-valued comparisons.
\begin{equation}
\label{eq:I-BWD}
\tag{I-BWD}
\begin{aligned}
\xi_{I}^{*} = \; &\text{min.} & & 
\xi \\
& \text{s.t.} & &  -\xi + a_{Bi}^{-} V(\mathbf{x}_{i}) \leq V(\mathbf{x}_{B}) \leq  a_{Bi}^{+}  V(\mathbf{x}_{i}) + \xi \quad \forall {i} \in R \setminus \{B \}  \, , \\
& & &   -\xi + a_{iW}^{-} V(\mathbf{x}_{W}) \leq V(\mathbf{x}_{i}) \leq  a_{iW}^{+} V(\mathbf{x}_{W}) + \xi \quad  \forall {i} \in R \setminus \{W \} \, , \\
&  & &  v_j (x_j^k ) \leq v_j (x_j^{k+1} ) \hspace{0.5cm} \forall j \in N \text{ and }  \forall k \, , \\
&  & & \sum_{j=1}^{n} v_j (\overline{x}_j )=1 \, , \\
& & & v_j (\underline{x}_j )=0  \hspace{0.5cm} \forall j \in N \, , \\
& & & \xi \geq 0 \, .
\end{aligned}
\end{equation}

Note that \ref{eq:I-BWD} is a generalization of \ref{eq:BWD} and when the preference information is given in the form of real numbers, the former collapses into the latter. Indeed, all the steps proposed to test the robustness of the real-valued BWD method (i.e., optimization problems \eqref{eq:BWDnec}, \eqref{eq:BWDmaxrank}, and \eqref{eq:BWDminrank}) can be extended to the interval case by replacing the constraints
\begin{equation}
\begin{aligned}
%\begin{cases}
    - \xi^{*} & \leq V(\mathbf{x}_B ) - a_{B i} V(\mathbf{x}_i) \leq \xi^{*}\hspace{0.2cm}  \forall {i} \in R \setminus \{B \} \, , \\ 
    -\xi^{*} &\leq V(\mathbf{x}_i ) - a_{i W} V(\mathbf{x}_W ) \leq \xi^{*} \hspace{0.2cm}  \forall {i} \in R \setminus \{ W \} \, .
%\end{cases}
\end{aligned}
\end{equation}
with the generalized interval constraints:
\begin{equation}
\begin{aligned}
 -\xi_{I}^{*} + a_{Bi}^{-} V(\mathbf{x}_{i}) & \leq V(\mathbf{x}_{B}) \leq a_{Bi}^{+} V(\mathbf{x}_{i}) + \xi_{I}^{*} \hspace{0.2cm}  \forall {i} \in R \setminus \{B \} \, , \\
 -\xi_{I}^{*} + a_{iW}^{-} V(\mathbf{x}_{W}) & \leq V(\mathbf{x}_{i}) \leq a_{iW}^{+} V(\mathbf{x}_{W}) + \xi_{I}^{*} \hspace{0.2cm}  \forall {i} \in R \setminus \{W \} \, .
\end{aligned}
\end{equation}
Indeed, BWD can be seen as a restriction of I-BWD. That is, given a problem, if all the real-valued preferences lie within the interval-valued ones, then $\xi^{*}_{I} \leq \xi^{*}$.

\begin{proposition}
%If we call $E$ the set of all pairs on which comparisons were elicited.
If $a_{Bi} \in [a_{Bi}^{-},a_{Bi}^{+}]$, $a_{iW} \in [a_{iW}^{-},a_{iW}^{+}]$ and the consequence vectors are the same for both models, then, $\xi^{*}_{I} \leq \xi^{*}$.
\end{proposition}

\begin{proof} Let us consider the case of the comparisons between the best and the others. In the I-BWD, the first family of constraints is a linearization of
\begin{equation}
\label{eq:xi_IBWD}
\xi_{I} \geq \max \left\{ V(\mathbf{x}_B)-a^{+}_{Bi}V(\mathbf{x}_{i}) , a^{-}_{Bi}V(\mathbf{x}_{i}) - V(\mathbf{x}_B), 0 \right\} \, .
 \end{equation}
In the BWD, if we consider the real value $a_{Bi}$ as a degenerate interval, i.e., $a_{Bi} \in [a_{Bi},a_{Bi}]$, then we can write
\begin{equation}
\label{eq:xi_BWD}
\xi \geq \max \left\{ V(\mathbf{x}_B)-a_{Bi}V(\mathbf{x}_{i}) , a_{Bi}V(\mathbf{x}_{i}) - V(\mathbf{x}_B), 0 \right\}.
 \end{equation}
 As we know that $a_{Bi}^{-}\leq a_{Bi} \leq a_{Bi}^{+}$ and $V(\mathbf{x}_{i})\geq 0$ for all $i \in R$ we deduce 
 \begin{align*}
  V(\mathbf{x}_B)-a^{+}_{Bi}V(\mathbf{x}_{i}) & \leq V(\mathbf{x}_B)-a_{Bi}V(\mathbf{x}_{i}) \, , \\
  a^{-}_{Bi}V(\mathbf{x}_{i}) - V(\mathbf{x}_B) & \leq a_{Bi}V(\mathbf{x}_{i}) - V(\mathbf{x}_B)\, .
 \end{align*}
and therefore the right-hand side of \eqref{eq:xi_BWD} cannot be smaller than the right-hand side of \eqref{eq:xi_IBWD}. A similar line of reasoning can be applied to the second family of constraints for both BWD and I-BWD. Since $\xi \geq 0$ holds, then all other constraints are the same, guaranteeing the same freedom in the definition of $V$. So, since we minimize both $\xi_{I}$ and $\xi$, we obtain $\xi_{I}^{*} \leq \xi^{*}$.
\end{proof}

The impact of uncertainty on the final decision is channeled through
two mechanisms. First, the feasible set of value functions induced by
the interval judgments is a (generally non-empty) polytope whose size
grows with the width of the input intervals; wider intervals therefore
translate into a larger set of compatible value functions. Second,
this set is projected onto the necessary preference relation
$\succ_{\mathcal{N}}$, which only retains the rankings that hold for
\emph{every} compatible value function. As a consequence, broader
uncertainty typically yields a sparser $\succ_{\mathcal{N}}$ (more
incomparabilities), whereas tighter intervals yield a denser relation,
approaching a complete order as $\xi^{*}_{I}\!\to\!0$. %This monotone behavior is empirically confirmed by the robustness analysis of Section~\ref{sec:case-study}.

Table~\ref{tab:comparison} positions BWD against its most closely related methods: AHP~\citep{Saaty1977}, BWM~\citep{Rezaei2015}, UTA/UTASTAR~\citep{JacquetSiskos1982,SiskosEtAl2016}, and DOR~\citep{BarbatiGrecoLami2024}. Three features jointly distinguish BWD: (i) it brings the cognitively efficient Best-Worst protocol into the disaggregation paradigm; (ii) it exploits the cardinal content of the comparisons to infer a full piecewise linear value function, whereas UTA relies on purely ordinal rankings; and (iii) it is the only disaggregation method equipped with a joint ordinal and cardinal consistency check ($\xi^{*}$, $OR$, $CR$).

\begin{table*}[!t]
\centering
\caption{Feature-by-feature comparison of BWD with related MCDA methods.}
\resizebox{\columnwidth}{!}{
\label{tab:comparison}
\renewcommand{\arraystretch}{1.25}
\begin{tabular}{@{}lccccc@{}}
\toprule
\textbf{Feature} & \textbf{AHP} & \textbf{BWM} & \textbf{UTA/UTASTAR} & \textbf{DOR} & \textbf{BWD (proposed)}\\
\midrule
Paradigm              & Aggregation                     & Aggregation             & Disaggregation           & Disaggregation              & Disaggregation\\
Elicitation target    & Criteria \& alternatives        & Criteria \& alternatives                & Ref.\ alternatives       & Ref.\ alternatives          & Ref.\ alternatives\\
Elicitation protocol  & Full pairwise matrix            & BO and OW vectors        & Holistic ranking         & Deck-of-Cards               & $A^{BO}$ and $A^{OW}$ vectors\\
Elicitation effort    & $n(n-1)/2$ pairs                & $2n-3$ pairs            & $r-1$ ordered pairs      & $r$ card assignments        & $2r-3$ pairs\\
Type of judgment      & Cardinal ratios                 & Cardinal ratios         & Ordinal only             & Ranking + intensity         & Cardinal ratios\\
Counter-anchoring     & $\times$                        & \checkmark\  & $\times$                 & $\times$                    & \checkmark\ \\
Consistency check     & \checkmark\ ($CR$, Saaty)       & \checkmark\ ($\xi^{*},OR,CR$) & $\times$                   & $\times$                         & \checkmark\ ($\xi^{*},OR,CR$)\\
Optimization model    & Eigenvector / LP                & LP / NLP                & LP                       & LP                          & LP\\
Main output           & Weights and priorities       & Weights and priorities  & Value function $V$       & $\succ_{\mathcal{N}},\succ_{\mathcal{P}}$ & $V$, $\succ_{\mathcal{N}}$, extreme ranks\\
\bottomrule
\end{tabular}
}
\end{table*}

\section{A case study in logistics performance evaluation}
\label{sec:case-study}

We tested the BWD on the data collected and used by \citet{RezaeiEtAl2018} to rank countries according to a composite indicator of logistics performance. In their study, data from the World Bank was used to assess the performance of countries worldwide with respect to the six criteria: Customs, Infrastructure, Services, Timeliness, Tracking and tracing, and International shipment. For ease of interpretation of the results, we focused on European countries.

We used the minimum and maximum values of each attribute as lower and upper levels of each attribute. That is, given the decision matrix $\mathbf{X}=(x_{ij})_{m \times n}$ and considering that all six criteria were originally expressed as benefit criteria, we defined $\underline{x}_{j}=\min_{i \in M} x_{ij}$ and $\overline{x}_{j}=\max_{i \in M} x_{ij}$.

\paragraph{Reference set selection and preference elicitation}
One splitting point was allowed for each criterion set and was positioned midway in the interval $[\underline{x}_{j},\overline{x}_{j}]$. Next, we solved the set covering-inspired optimization problem \eqref{eq:OF}--\eqref{eq:c3} to find a mutually non-dominated minimum covering set. This set contained only three countries and was taken as a good initial suggestion. Given the possibility of asking for more comparisons, we added two more alternatives to better explore the lower part of the ranking. The final reference set $R$ was defined as
\[
R=\{ \text{Estonia,\,Hungary,\,Latvia}\} \cup \{ \text{Greece,\,Moldova} \} \subset M
\]

By adding the coverage provided by each alternative for each criterion segment, we arrived at the following matrix:
\begin{equation*}
 \mathbf{A}^{R} = \left(\sum_{i \in R}  a_{ijk} \right)_{j\in N,k \in \{ 1,2\}} =
%        \begin{pmatrix}
%    3 & 2 \\
%    2 & 3 \\
%    3 & 2 \\
%    4 & 1 \\
%    2 & 3 \\
%    2 & 3 
%    \end{pmatrix}
          \begin{pmatrix}
    3 & 2 & 3 & 4 & 2 & 2 \\
    2 & 3 & 2 & 1 & 3 & 3
    \end{pmatrix}^{T} \, .
\end{equation*}
Since each entry of the matrix $\mathbf{A}^{R}$ is at least one, we know that the alternatives indexed in $R$ cover all attribute levels. We can see that, as expected, the first three alternatives are sufficient to cover all criteria levels and that none of the first three alternatives dominates the others. This corroborates the validity of the optimization problem \eqref{eq:OF}--\eqref{eq:c3}. Next, an expert used the 1--9 scale and provided preference information on the reference set $R$ by means of the $A^{BO}$ and $A^{OW}$ sets presented in Table \ref{tab:ABO&AOW}.
\begin{table}[htb]
    \centering
        \caption{Best-to-Others ($A^{BO}$) and Others-to-Worst ($A^{OW}$) pairwise comparison vectors.}
    \begin{tabular}{llllll}
    \toprule
      &Estonia& Hungary & Latvia & Greece & Moldova \\
    \midrule
       $A^{BO}$ ($B$=Est.)&1& 3& 4& 5& 8\\
 $A^{OW}$ ($W$=Mold.)& 8& 5& 3& 4&1\\
 \bottomrule
    \end{tabular}
    \label{tab:ABO&AOW}
\end{table}
Next, we verified the consistency of these preferences to detect possibly inaccurate and unreliable preference information. This involves checking both ordinal and cardinal consistency, and making adjustments if necessary.

\paragraph{Consistency analysis}

First, we assessed the cardinal consistency by calculating the input-based $CR$ according to formula \eqref{eq:cr}. In this case, $CR=0.214$, which is less than the threshold (see \citep{liang2020}) of $0.284$ for 5 alternatives and $a_{BW}=8$. This indicates that the cardinal inconsistency of the provided preferences is tolerable, and we can proceed to the next step.

Using $OR$ as formulated in eq.~\eqref{eq:or}, we found $OR=0.2$, indicating a violation of ordinal consistency. The local $OR_i$ values, shown in Table \ref{tab:ordinal_check}, help locate specific inconsistencies. In this case, the preferences provided by the expert on Latvia are not ordinally consistent with those expressed on Greece.

\begin{table}
\centering
\caption{The ordinal consistency check table: ``1''  indicates an ordinal inconsistency between the corresponding alternatives, while ``0'' indicates consistency.}
\begin{tabular}{l l l l l l l}
\toprule
  &   Estonia&   Hungary &   Latvia &   Greece &   Moldova &   $OR_i$\\
\midrule
  Estonia& 0 & 0 & 0 & 0 & 0 & 0 \\
  Hungary & 0 & 0 & 0 & 0 & 0 & 0\\
  Latvia & 0 & 0 & 0 & 1& 0 & 0.2\\
  Greece & 0 & 0 & 1& 0 & 0 & 0.2\\
  Moldova & 0 & 0 & 0 & 0 & 0 & 0 \\
  \bottomrule
\end{tabular}
    \label{tab:ordinal_check}
\end{table}
After identifying and locating inconsistencies in the preferences, we re-engaged with the expert to review and reconsider her judgments. Upon reflection, the expert acknowledged the inconsistencies and proceeded to modify the preferences using the consistency improvement process from the work of \citet{liang2022}. The improving ranges represent the feasible adjustment intervals for each preference parameter that would eliminate inconsistencies while maintaining compatibility with the expert's other stated preferences. %Specifically, these ranges indicate the minimum and maximum values that each parameter can take to achieve consistency, providing the expert with clear guidance on how much adjustment is needed. Parameters with narrow improving ranges require precise adjustments, while those with wider ranges offer more flexibility in the revision process. }
Using the same approach, we derived the admissible ranges for $A^{BO}$ and $A^{OW}$, as shown in Table \ref{tab:range}, and the visualization of the improving ranges, shown in Fig.~\ref{fig:improving}.  {The visualization clearly illustrates which parameters require adjustment (those whose current values fall outside their improving ranges) and the extent of modification needed to restore consistency. }

\begin{table}[htb]
\centering
\caption{The admissible ranges for revision.}
%\resizebox{\columnwidth}{!}{
\begin{tabular}{l l l l l l l}
\toprule
Alternatives &   & Estonia & Hungary & Latvia & Greece & Moldova \\
\midrule
Original &   & [1, 8] & [1, 8] & [1, 8] & [1, 8] & [1, 8] \\
\midrule
\multirow{2}{*}{\makecell[l]{Acceptable ranges}}
  & $A^{BO}$ & [1, 2.99] & [1, 4.78] & [1, 7.97] & [1, 5.98] & [1, 8] \\
  & $A^{OW}$ & [1, 8] & [1, 7.97] & [1, 5.98] & [1, 4.78] & [1, 2.99] \\
\multirow{2}{*}{\makecell[l]{Improving ranges}}
  & $A^{BO}$ & [1, 1] & [1.6, 3] & [2.67, 4] & [2, 5] & [8, 8] \\
  & $A^{OW}$ & [8, 8] & [3.6, 5] & [1.67, 3] & [1, 4] & [1, 1] \\
\bottomrule
\end{tabular}
%}
\label{tab:range}
\end{table}

\begin{figure}[htp]
            \centering
            \includegraphics[width=0.45\linewidth]{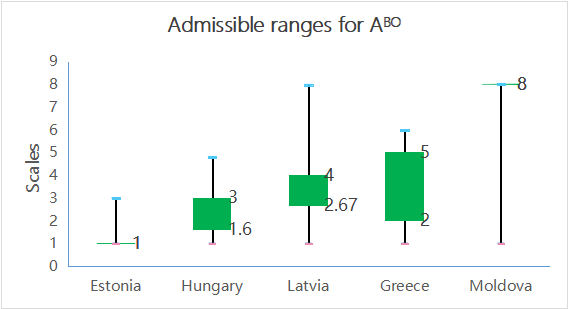}
            \hfill
            \includegraphics[width=0.45\linewidth]{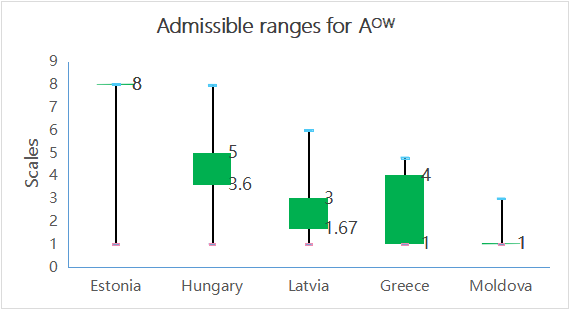}
            \caption{Improving ranges, in green, for $A^{BO}$ and $A^{OW}$.}
            \label{fig:improving}
\end{figure}
With the help of the visualization of the ranges for improvement, the expert can adjust the inconsistent comparisons. After consideration, the expert made some minor revisions and the modified sets $A^{BO}$ and $A^{OW}$ are reported in Table \ref{tab:ABO&AOW2}.

\begin{table}[htb]
    \centering
        \caption{Revised $A^{BO}$ and $A^{OW}$ pairwise comparison vectors.}
    \begin{tabular}{llllll}
    \toprule
      &Estonia& Hungary & Latvia & Greece & Moldova \\
    \midrule
       $A^{BO}$ ($B$=Est.)&1& 3& 4& 4& 8\\
 $A^{OW}$ ($W$=Mold.)& 8& 5& 3& 3&1\\
 \bottomrule
    \end{tabular}
    \label{tab:ABO&AOW2}
\end{table}
%After revising the preferences, we rechecked their consistency using both the ordinal and cardinal consistency indices.
The revised preferences showed $OR=0$, indicating full ordinal consistency, and $CR=0.125$, which is still below the threshold of $0.284$, confirming that the judgments are now acceptably consistent and represent an improvement over the initial preferences. 

\paragraph{Optimization results and compatibility}

Preferences were used as inputs in the BWD optimization problem, and an optimal value $\xi^{*}=0.03$ was found. This means that even the best-fitting model $V$ cannot comply exactly with the expert's judgments. This should be allowed, since the preferences of an expert expressed as real numbers can hardly ever fully reflect an underlying piecewise linear model.

The full analysis revealed that $\succ_{\mathcal{N}}$ is almost a total order relation. This is confirmed in Fig.~\ref{fig:ranks}, showing very low variability in the set of rankings obtained through optimization problems \eqref{eq:BWDmaxrank} and \eqref{eq:BWDminrank}. Only one swap is possible, between Greece and Slovenia. The results shown in Fig.~\ref{fig:ranks} seem to hint at a good precision of the model, and further tests that we conducted, but do not report, showed that incompatibility, i.e., $\xi^{*}>0$, often leads to stable rankings. Note also that the presence of alternative rankings corroborates that both BWD and I-BWD can lead to multiple optimal solutions.

 \vspace{1cm}
\begin{figure*}[htb]
    \centering
    \vspace{0.3cm}
    \includegraphics[trim=0 0 7mm 5mm, height=10.3cm]{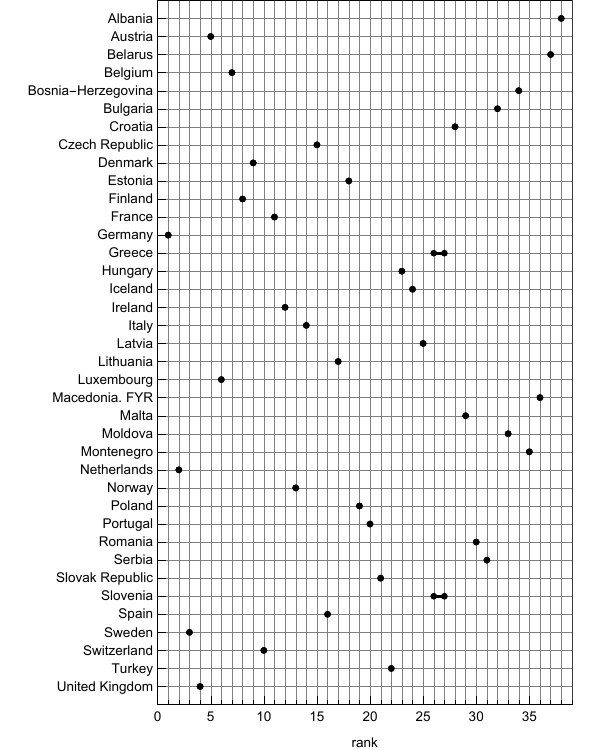}
    \hspace{-0.5cm}
    \includegraphics[trim=0 0 7mm 5mm, height=10.3cm]{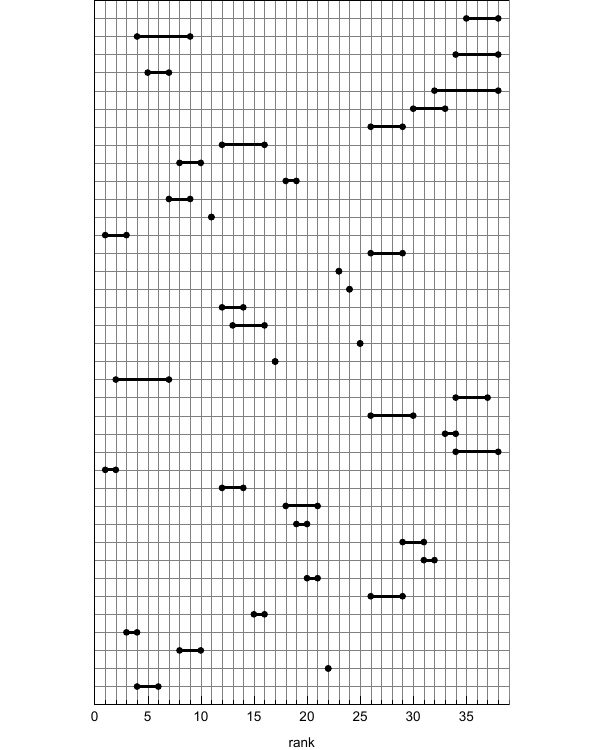}
    \caption{Rankings of European countries based on BWD (on the left) and I-BWD (on the right).}
    \label{fig:ranks}
\end{figure*}

\paragraph{Comparison with UTASTAR and DOR}

We compared BWD with UTASTAR \citep{JacquetSiskos1982} and DOR \citep{BarbatiGrecoLami2024} on the same case study, using the same reference set, decision matrix, and piecewise linear value functions with one splitting point per criterion.

\textit{Input derivation.}
UTASTAR requires only an ordinal ranking, extracted directly from both $A^{BO}$ and $A^{OW}$: Estonia $\succ$ Hungary $\succ$ Latvia $\sim$ Greece $\succ$ Moldova.
DOR requires a ranking augmented with blank cards encoding preference intensity. Since $a_{iW} = V(\mathbf{x}_i)/V(\mathbf{x}_W)$, the $A^{OW}$ ratios yield implied values Mol${}=1$, Gre${}=$Lat${}=3$, Hun${}=5$, Est${}=8$. Converting gaps to blank cards: $e_0{=}0$, $e_1{=}1$ (Mol $\to$ \{Gre, Lat\}), $e_2{=}1$ (\{Gre, Lat\} $\to$ Hun), $e_3{=}2$ (Hun $\to$ Est), giving $\nu$-values $\nu(\text{Mol}){=}1$, $\nu(\text{Gre}){=}\nu(\text{Lat}){=}3$, $\nu(\text{Hun}){=}5$, $\nu(\text{Est}){=}8$. We adopted the sum-of-deviations formulation (eq.~(7) in \citep{BarbatiGrecoLami2024}) with a piecewise linear additive value function.

\textit{Model fitting.}
BWD yields $\xi^* = 0.03$, a small but nonzero incompatibility, while both UTASTAR ($\theta^* = 0$) and DOR ($\sum \sigma = 0$, $k = 0.016$) achieve a perfect fit.
%This is expected: UTASTAR imposes only ordinal constraints, and DOR's $\nu$-values derived from the ratios are inherently smooth.
The nonzero $\xi^*$ of BWD reflects the additional stringency of simultaneously fitting two independent comparison vectors, which is precisely the mechanism enabling its consistency analysis.

\textit{Implied weights and rank agreement.}
All three methods agree that Services is the most important criterion, with a weight of $0.505$ (BWD), $0.428$ (UTASTAR), and $0.628$ (DOR). BWD and DOR also assign substantial weight to Customs ($0.454$ and $0.291$, respectively), while the stabilized UTASTAR distributes the remaining importance more broadly across the other criteria. In terms of rank agreement on the full set of 38 countries, the three methods are highly concordant: Kendall $\tau = 0.960$ for BWD and UTASTAR, $0.940$ for BWD and DOR, and $0.933$ for UTASTAR and DOR, with a maximum rank difference of only 3 positions and an average rank difference below $1.2$ in all pairs.

Applying each value function to all 38 countries, the top-3 countries are identical across methods (Germany, Netherlands, Sweden) and ranks 4--10 consist of the same seven countries with only minor positional swaps, as shown in Table~\ref{tab:top10_comparison}. For the reference set, all methods place Estonia around rank 18--20, Hungary 21--23, Latvia and Greece 25--26, and Moldova 33--34, confirming that the underlying preference structure is preserved across elicitation formats.

\begin{table}[htb]
\centering
\footnotesize
\caption{Top-10 countries ranked by BWD, with the corresponding ranks and value function scores obtained by UTASTAR and DOR.}
\label{tab:top10_comparison}
\renewcommand{\arraystretch}{1.1}
\begin{tabular}{lcccccc}
\toprule
\textbf{Country} & \textbf{BWD} & \textbf{UTA} & \textbf{DOR} \\
\midrule
Germany        & 1  & 1  & 1  \\
Netherlands    & 2  & 2  & 2  \\
Sweden         & 3  & 3  & 3  \\
United Kingdom & 4  & 7  & 5  \\
Austria        & 5  & 5  & 4  \\
Luxembourg     & 6  & 4  & 7 \\
Belgium        & 7  & 6  & 6  \\
Finland        & 8  & 9  & 10 \\
Denmark        & 9  & 10 & 8 \\
Switzerland    & 10 & 8  & 9 \\
\bottomrule
\end{tabular}
\end{table}

\textit{Discussion.}
Two additional observations emerge. %First, all pairwise Kendall $\tau \geq 0.93$, with maximum rank differences of only 3 positions. This provides strong evidence of convergent validity: the expert's preferences are captured reliably regardless of the elicitation format.
First, the UTASTAR comparison highlights the importance of stability analysis in pure ordinal-input methods: without it, the LP returns an extreme solution that degenerates the lower part of the ranking. BWD avoids this problem structurally, as its two comparison vectors ($A^{BO}$ and $A^{OW}$) impose tighter constraints on the feasible set. Second, consistency analysis remains a distinctive advantage of BWD: the ordinal inconsistency between Latvia and Greece detected by the $OR$ index would go unnoticed in both UTASTAR (which takes rankings as given) and DOR (which assumes $\nu$-values are reliable).

\paragraph{Interval-BWD}

We also tested the I-BWD on the same case study, together with the analysis of dominance relations and maximum and minimum ranks. In this case, the expert information is summarized in Table \ref{tab:ABO&AOW2_intervals}. %Note that the improving ranges shown in Table \ref{tab:range} can be used to guide experts to extend their original real-valued preferences to interval ones in a way that may decrease their inconsistency.
\begin{table}[htb]
    \centering
        \caption{The two sets of interval-valued comparisons: $\tilde{A}^{BO}$ and $\tilde{A}^{OW}$.}
    \begin{tabular}{llllll}
    \toprule
      &Estonia& Hungary & Latvia & Greece & Moldova \\
      \midrule
       $\tilde{A}^{BO}$ ($B$=Est.) & $[1,1]$ & $[2,3]$ & $[3,4]$ & $[4,5]$& $[7,9]$\\
 $\tilde{A}^{OW}$ ($W$=Mold.)& $[7,9]$ & $[4,5]$& $[2,4]$ &  $[3,4]$ & $[1,1]$ \\
 \bottomrule
    \end{tabular}
\label{tab:ABO&AOW2_intervals}
\end{table}

In the case of I-BWD, the imprecise information provided by the expert was sufficient to achieve full compatibility between the preferences and the underlying model, i.e., $\xi^{*}_{I}=0$. Consequently, the necessary relation $\succ_{\mathcal{N}}$ is less informative, as shown in the Hasse diagram in Fig.~\ref{fig:Hasse_intervals}.

\begin{figure}[th]
    \centering    \includegraphics[width=1\columnwidth]{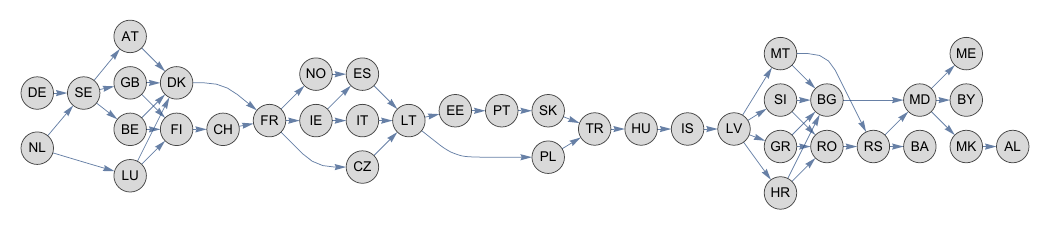}
    \caption{Necessary ranking $\succ_{\mathcal{N}}$ of European countries obtained with the interval-valued BWD. Countries are labeled using the ISO 3166-1 alpha-2 standard.}
    \label{fig:Hasse_intervals}
\end{figure}

\paragraph{Sensitivity analysis}

Results of the extreme ranking analysis are presented in the right-hand side of Fig.~\ref{fig:ranks} and confirm the previous intuition on the much greater variability of the attainable ranks. To quantify the imprecision, we used the heuristic index
\[
U=\frac{1}{m} \sum_{i \in M} \frac{\underline{r}_{i}-\overline{r}_{i}}{m-1} \in [0,1] \, ,
\]
where each term of the sum is the ``spread'' of the possible ranks that the $i$th alternative can attain divided by the maximum possible range, going from $1$, the best rank, to $m$, the lowest possible attainable rank. The scaling factor preceding the sum normalizes the quantity into the interval $[0,1]$. In the case of the I-BWD, we have $U = 0.05938$, whereas, for the real-valued BWD, it was $U=0.00135$.

\paragraph{Robustness analysis}

To assess the robustness of the I-BWD, we performed a sensitivity analysis by modifying part of the input information. In particular, two intervals were tightened and two were widened, as reported in Table~\ref{tab:revised_intervals}.

\begin{table}[htb]
    \centering
    \caption{Revised interval-valued comparisons: modified values are in boldface.}
    \begin{tabular}{llllll}
    \toprule
      & Estonia & Hungary & Latvia & Greece & Moldova \\
      \midrule
       $\tilde{A}^{BO}$ ($B$=Est.) 
       & $[1,1]$ & $[2,3]$ & $[3,\mathbf{5}]$ & $[4,5]$ & $[\mathbf{8},9]$\\
       $\tilde{A}^{OW}$ ($W$=Mold.) 
       & $[\mathbf{8},9]$ & $[\mathbf{3},5]$ & $[2,\mathbf{3}]$ & $[3,4]$ & $[1,1]$ \\
       \bottomrule
    \end{tabular}
\label{tab:revised_intervals}
\end{table}

These revised inputs result in the necessary preference relation shown in Fig.~\ref{fig:Hasse2}. A comparison between Figs.~\ref{fig:Hasse_intervals} and \ref{fig:Hasse2} reveals that the two Hasse diagrams represent different, yet structurally similar, order relations.

\begin{figure}[ht]
    \centering
    \includegraphics[width=1\columnwidth,trim=0 2cm 0 0, clip ]{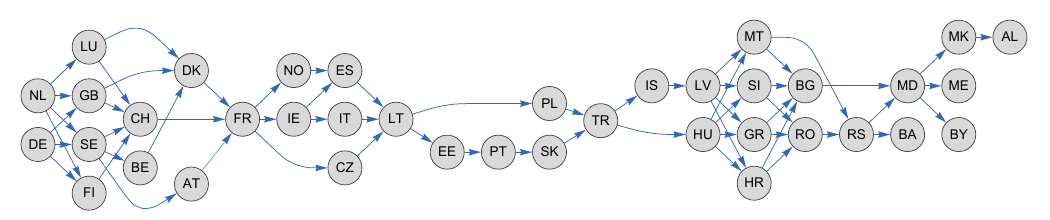}  \caption{Necessary ranking $\succ_{\mathcal{N}}$ of European countries obtained with the interval-valued BWD using the revised inputs in Table \ref{tab:revised_intervals}.}
    \label{fig:Hasse2}
\end{figure}

This behavior suggests that the I-BWD model is stable but not rigid: moderate perturbations of the input intervals lead to limited changes in the induced necessary preference relation, while preserving its main structural features.  
Further experiments indicate that this pattern typically occurs when the model is fully compatible, i.e., when $\xi^{*}_{I}=0$. In this case, the feasible set of value functions consistent with the input information forms a non-empty polytope, and small variations in the intervals only moderately reshape this feasible region. As a consequence, the necessary preference relation, defined as the intersection of all compatible value functions, changes smoothly. By contrast, when $\xi^{*}_{I} > 0$, the model requires compromise to restore feasibility. This reduces the effective variability of compatible value functions and tends to produce a more determinate ranking. Empirically, we observe that in such cases the necessary preference relation is more likely to approach a linear (complete) order. Hence, the compatibility index $\xi^{*}_{I}$ plays a structural role in determining the degree of incomparability among alternatives.

Additional caution is required when specifying the number of splitting points in the marginal value functions. %The number of splitting points directly determines the flexibility of the additive value model and affects the determinacy of the resulting ranking.
In our case study, if no mid-value splitting points are introduced, the marginal value functions are restricted to be linear; we obtain $\xi^{*}_{I} > 0$; and the necessary preference relation $\succ_{\mathcal{N}}$ becomes a linear order.
%This outcome is consistent with the reduced flexibility of the model.
Conversely, if two splitting points are allowed for each marginal value function, the model becomes substantially more flexible.
The resulting necessary ranking is shown in Fig.~\ref{fig:necessary2}.% As expected, we observe a significantly higher degree of incomparability, especially among lower-ranked countries.
\begin{figure}[ht]
    \centering
    \includegraphics[width=1\columnwidth]{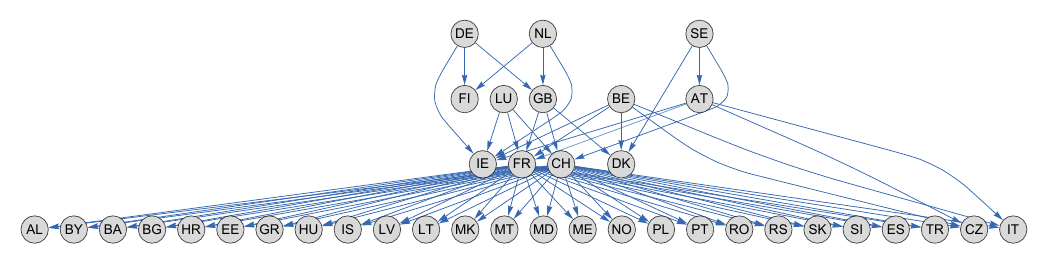}
    \caption{Necessary ranking $\succ_{\mathcal{N}}$ of European countries obtained with the interval-valued BWD and two splitting points for each marginal value function.}
    \label{fig:necessary2}
\end{figure}

Nevertheless, even in this more flexible specification, a limited number of alternatives (five countries) have no predecessors in the necessary relation. This provides reasonable confidence that the best alternative is very likely to belong to this subset. Overall, the number of splitting points emerges as a key modeling parameter and can also be used as a sensitivity parameter to analyze the robustness of the conclusions with respect to model flexibility.

\paragraph{Computational feasibility}

The analysis for the I-BWD, including the necessary relation $\succ_{\mathcal{N}}$ represented in Fig.~\ref{fig:Hasse_intervals} and the maximum and minimum ranks in Fig.~\ref{fig:ranks}, was carried out with Wolfram Mathematica 14.0.0, with a total computational time of 16.65 seconds on an Intel(R) Core(TM) i7-8550U CPU @ 1.80GHz with 16 GB RAM. Optimization problems were solved using Gurobi 11.0.1. If the problem was run for the entire set considered by \citet{RezaeiEtAl2018} (165 countries), then the computational time would still remain under 3 minutes. %These results seem to support the computational feasibility of the proposed approach. Furthermore, out of the 16.65 seconds employed for the case of European countries, 13.53 seconds were spent in solving problem \eqref{eq:BWDnec}, which, as it must be solved for all distinct pairs of indices $(p,q)$, has quadratic complexity with respect to $m$. By contrast, the optimization problems used to find the best and worst attainable ranks have linear complexity with respect to $m$. Hence, if the optimization problems used to find $\succ_{\mathcal{N}}$ were omitted, then the entire procedure would have much greater scalability.

\section{Conclusions}
\label{sec:conclusions}

The BWD offers an advancement in the field of MCDA by combining the principles of the BWM with the disaggregation approach. This method addresses the challenges inherent in traditional aggregation and disaggregation techniques by leveraging the strengths of both to create a robust and reliable decision-making framework. It is a valuable tool in scenarios where experts have a clear understanding of alternatives but may not have predefined value functions. BWD's core advantage lies in its ability to derive attribute-specific value functions directly from the experts' two pairwise comparison vectors of alternatives. %This process ensures that the resulting weights and value functions align closely with the actual preferences of the experts, leading to more reliable outcomes.
BWD reduces the cognitive load on experts and improves the reliability of derived models by simplifying the comparison process and ensuring consistency through admissible ranges for iterative visual refinements. BWD's structured approach to preference elicitation and consistency improvement enhances the transparency of the decision process. We proposed a method to select the reference set that ensures comprehensive coverage of attribute levels, avoids Pareto dominance among alternatives, and minimizes cognitive load on experts.
%This approach improves the representativeness and effectiveness of the reference set, contributing to more accurate and reliable preference modeling.
Furthermore, we extended the BWD method to handle interval-valued preferences, allowing for the accommodation of uncertainty or imprecise information in expert judgments, which broadens the applicability of BWD in real-world decision-making scenarios.

From a systems-engineering perspective, the relevance of BWD extends beyond the technical contributions listed above. In recent years, decision-support systems in evidence-sensitive domains have moved towards a model in which the analyst is expected to deliver not only a recommended ranking, but also a justification of that ranking that is robust, traceable, and intelligible to non-technical stakeholders; a requirement that aligns with the broader push for explainable decision models \citep{Rudin2019}. BWD is well positioned to support this shift: the necessary preference relation $\succ_{\mathcal{N}}$ identifies the conclusions that are invariant across all value functions compatible with the expert's judgments, the consistency analysis documents how and why the elicited preferences were revised before the final model was fitted, and the extreme-rank bounds quantify the residual sensitivity of each alternative's position to admissible variations of the value function. Our logistics performance case study in Section~\ref{sec:case-study} illustrates this workflow end-to-end. The same structure (a small set of well-known alternatives, expert judgments that are most naturally expressed in holistic terms, and stakeholders who require transparent robustness information) arises in adjacent application domains, including sustainability and ESG assessment of countries, firms and supplier or third-party logistics evaluation in supply-chain management, where BWD can be embedded as a preference-elicitation and value-function-inference module within a broader decision-support pipeline.

Despite these contributions, BWD still has limitations that suggest future research directions. The current formulation is designed for a single decision-maker and a static decision setting, and its computational demands may grow for substantially larger sets of criteria and alternatives. Future work should test BWD in real-world case studies, compare it with classical methods, extend it to group decision-making settings \citep{QinMaPedrycz2023,ZhangEtAl2024}, and adapt it to dynamic environments. It would also be useful to investigate ways to reduce biases in expert judgments, for example through behavioral decision-theory techniques, machine-learning support, or alternative elicitation scales such as those used in GRIP \citep{figueira2009}.

\section{Appendix}
\label{sec:Appendix}

For the reader's convenience, Table~\ref{tab:notation} lists the
notation used throughout the paper.

\begin{table}[!t]
\centering
\footnotesize
\caption{Main notation used in the paper.}
\label{tab:notation}
\renewcommand{\arraystretch}{1.1}
\begin{tabular}{@{}ll@{}}
\toprule
\textbf{Symbol} & \textbf{Description}\\
\midrule
$M,\,N$                  & Index sets of alternatives ($|M|{=}m$) and criteria ($|N|{=}n$)\\
$\mathbf{x}_i$           & Consequence vector of the $i$th alternative\\
$x_j,\;X_j$              & Level and domain of the $j$th attribute\\
$v_j:X_j\!\to\![0,1]$    & Marginal (attribute-specific) value function\\
$V$                      & Overall additive value function\\
$w_j$                    & Weight of the $j$-th criterion\\
$R$                      & Index set of the reference alternatives, $|R|{=}r$\\
$B,\,W$                  & Indices of the best and worst alternative in $R$\\
$a_{Bi},\,a_{iW}$        & Best-to-Others and Others-to-Worst comparisons\\
$\tilde{a}_{Bi},\,\tilde{a}_{iW}$ & Interval-valued counterparts (I-BWD)\\
$\xi^{*},\,\xi^{*}_{I}$  & Compatibility index of BWD and I-BWD\\
$OR,\,CR$                & Ordinal and cardinal consistency ratios\\
$\succ_{\mathcal{N}}$           & Necessary preference relation\\
\bottomrule
\end{tabular}
\end{table}

{ \footnotesize
\bibliographystyle{plainnat}
\bibliography{biblio}
}
\end{document}